\documentclass[a4paper,11pt]{amsart}

\usepackage{cmbright}

\usepackage{amsmath,amsthm,amssymb,eucal, amscd, mathrsfs,amsfonts}

\usepackage{hyperref}
\usepackage{array}

\usepackage{latexsym}

\usepackage[all]{xy}

 \def\commutatif{\ar@{}[rd]|{\circlearrowleft}}
 
\newcommand{\eq}[1][r]
   {\ar@<-3pt>@{-}[#1]
    \ar@<-1pt>@{}[#1]|<{}="gauche"
    \ar@<+0pt>@{}[#1]|-{}="milieu"
    \ar@<+1pt>@{}[#1]|>{}="droite"
    \ar@/^2pt/@{-}"gauche";"milieu"
    \ar@/_2pt/@{-}"milieu";"droite"}
    
    \def\dar[#1]{\ar@<2pt>[#1]\ar@<-2pt>[#1]}
  \entrymodifiers={!!<0pt,0.7ex>+} 

\newtheorem{thm}{Theorem}[section]
\newtheorem{pro}[thm]{Proposition}
\newtheorem{lem}[thm]{Lemma}

\newtheorem{rem}[thm]{Remark}

\theoremstyle{definition}
\newtheorem{df}[thm]{Definition}

\newtheorem{ex}[thm]{Example}
\newtheorem{exs}[thm]{Examples}
\allowdisplaybreaks
\numberwithin{equation}{section}

\newcommand{\cA}{{\mathcal A}}
\newcommand{\cB}{{\mathcal B}}
\newcommand{\cC}{{\mathcal C}}

\newcommand{\cF}{{\mathcal F}}
\newcommand{\cG}{{\mathcal G}}
\newcommand{\cH}{{\mathcal H}}

\newcommand{\cK}{{\mathcal K}}
\newcommand{\cL}{{\mathcal L}}

\newcommand{\cU}{{\mathcal U}}

\newcommand{\CC}{{\mathbb C}}

\newcommand{\NN}{{\mathbb N}}

\newcommand{\PP}{{\mathbb P}}

\newcommand{\RR}{{\mathbb R}}
\newcommand{\bbS}{{\mathbb S}}
\newcommand{\uc}{{\mathbb S}^1}
\newcommand{\TT}{{\mathbb T}}

\newcommand{\ZZ}{{\mathbb Z}}

\newcommand{\bfS}{{\mathbf S}}

\newcommand{\id}{{\mathbf 1}}

\newcommand{\Id}{{\text{id}}}

\newcommand{\wGa}{{\widetilde \Gamma}}

\newcommand{\U}{{\rm U}}
\newcommand{\wU}{{\hat{\rm U}}}
\newcommand{\wPU}{{\what{\rm PU}}}

\newcommand{\wRExt}{{\what{\operatorname{\textrm{Ext}R}}}}

\newcommand{\Rext}{{\what{\operatorname{\textrm{ext}R}}}}
\newcommand{\wRBr}{{\what{\operatorname{\textrm{Br}R}}}}

\newcommand{\wK}{\what{\cK}}
\newcommand{\wKK}{\hat{\mathbb{K}}}

\newcommand{\PicR}{{\operatorname{\textrm{Pic}R}}}
\newcommand{\wBr}{\what{{\text{\textrm{Br}}}}}

\newcommand{\Gpdo}{{\mathcal{G}^{(0)}}}
\newcommand{\grpd}{\xymatrix{\cG \dar[r]^r_s & \Gpdo}}

\newcommand{\Gamo}{{\Gamma^{(0)}}}

\newcommand{\what}{\widehat}

\newcommand{\fr}{{\mathfrak r}}
\newcommand{\fs}{{\mathfrak s}}

\newcommand{\RG}{{\mathfrak {RG}}}

\newcommand{\To}{\longrightarrow}

\newcommand{\mto}{\longmapsto}

\newcommand{\cstar}{C^{\ast}}
\newcommand{\Ga}{\Gamma}

\newcommand{\del}{\delta}
\newcommand{\al}{\alpha}

\newcommand{\ve}{\varepsilon}
\newcommand{\vp}{\varphi}
\newcommand{\g}{\gamma}

\newcommand{\Lam}{\Lambda}
\newcommand{\Isom}{\operatorname{Isom}}

\def\<{\langle}
\def\>{\rangle}
\let\ipscriptstyle=\scriptscriptstyle
\def\lipsqueeze{{\mskip -3.0mu}}
\def\ripsqueeze{{\mskip -3.0mu}}
\def\ipcomma{\nobreak\mathrel{,}\nobreak}
\newbox\ipstrutbox
\setbox\ipstrutbox=\hbox{\vrule height8.5pt
width 0pt}
\def\ipstrut{\copy\ipstrutbox}
\def\lip#1<#2,#3>{\mathopen{\relax_{\ipstrut\ipscriptstyle{
#1}}\lipsqueeze
\langle} #2\ipcomma #3 \rangle}
\def\blip#1<#2,#3>{\mathopen{\relax_{\ipstrut
\ipscriptstyle{ #1}}\lipsqueeze\bigl\langle} #2\ipcomma #3 \bigr\rangle}
\def\rip#1<#2,#3>{\langle #2\ipcomma #3
\rangle_{\ripsqueeze\ipstrut\ipscriptstyle{#1}}}
\def\brip#1<#2,#3>{\bigl\langle #2\ipcomma #3
\bigr\rangle_{\ripsqueeze\ipstrut\ipscriptstyle{#1}}}
\def\angsqueeze{\mskip -6mu}
\def\smangsqueeze{\mskip -3.7mu}
\def\trip#1<#2,#3>{\langle\smangsqueeze\langle #2\ipcomma #3
\rangle\smangsqueeze\rangle_{\ripsqueeze\ipstrut\ipscriptstyle{#1}}}
\def\btrip#1<#2,#3>{\bigl\langle\angsqueeze\bigl\langle #2\ipcomma
#3
\bigr\rangle
\angsqueeze\bigr\rangle_{\ripsqueeze\ipstrut\ipscriptstyle{#1}}}
\def\tlip#1<#2,#3>{\mathopen{\relax_{\ipstrut\ipscriptstyle{
#1}}\lipsqueeze \langle\smangsqueeze\langle} #2\ipcomma #3
\rangle\smangsqueeze\rangle}
\def\btlip#1<#2,#3>{\mathopen{\relax_{\ipstrut\ipscriptstyle{
#1}}\lipsqueeze
\bigl\langle\angsqueeze\bigl\langle} #2\ipcomma #3
\bigr\rangle\angsqueeze\bigr\rangle}

\def\ip(#1|#2){(#1\mid #2)}
\def\bip(#1|#2){\bigl(#1 \mid #2\bigr)}
\def\Bip(#1|#2){\Bigl( #1 \bigm| #2 \Bigr)}
\def\h[#1,#2]{[#1,#2]_{H}}

\def\ipp(#1|#2){\ip({#1}|{#2})_{\pi}}

\newcommand{\Hom}{{\operatorname{Hom}}}

\title{Twistings of $KR$ for Real groupoids}
\author{El-ka\"{i}oum M. Moutuou}
\address{Universit\'e Paul Verlaine - Metz}
\email{moutuou@univ-metz.fr}

\date{}

\begin{document}

\maketitle

\begin{abstract}

$B$-fields over a groupoid with involution are defined as Real graded Dixmier-Douady bundles. We use these to introduce the Real graded Brauer group $\wRBr_\ast(\cG)$ which constitutes the set of twistings for Atiyah's $KR$-functor in the category of locally compact groupoids with involutions. We interpret this group in terms of groupoid extensions and elements of some equivariant \v{C}ech cohomology theory. The construction of the twisted $KR$-functor is outlined. 
\end{abstract}

\section*{Introduction}

A space $X$ equipped with an involution $\tau:X\To X$ (\emph{i.e.} $\tau$ is a homeomorphism such that $\tau^2=\id$) is called a Real space. A Real vector bundle over $X$ is a complex vector bundle $E\To X$ which is itself a Real space such that the projection is equivariant, and the induced isomorphism $E_x\To E_{\tau(x)}$ is conjugate linear. The group $KR(X)$ was defined by Atiyah~\cite{At1} to be the Grothendieck group of isomorphism classes of Real vector bundles over $X$. It was argued in~\cite[\S5]{Wit} that in a spacetime manifold with a $2$-periodic symmetry $\tau$, $D$-brane charges correspond, in the case of a vanishing NS $B$-field, to Real vector bundles on the Real space $X$; while in the presence of a non-vanishing flat $B$-field, $D$-branes take values in twisted $K$-theory~\cite{BM}. It is moreover known~\cite{DFM} that a NS $B$-field corresponds to $3$-integral cohomology class $\omega$ such that $\tau^\ast\omega =-\omega$, hence an element in some appropriate \emph{cohomology theory} $HR^\ast$ for Real spaces to be defined. It is then natural to expect that in the presence of a symmetry $\tau$ and of $B$-fields, $D$-branes should take values in a twisted version of $KR$-theory. The purpose of this note is to bring this idea to the more general context of groupoids.

The paper is organised as follows. We give the notion of \emph{Real groupoids} and illustrate it by simple examples in \S1. In \S2, we introduce a kind of equivariant cohomology theory for Real groupoids and in \S3 we adapt the theory of non-abelian extensions to the Real case. In \S4, we classify involutions on graded algebras of compact operators which may be thought of as the universal elements in the study of Real graded Dixmier-Douady bundles we introduce. \S5 is devoted to the study of $B$-fields in terms of cohomology classes. We close these notes by defining \emph{twisted Real $K$-theory} of Real groupoids and giving some of its basic properties.


\section{Real groupoids}

Recall that a topological groupoid $\grpd$ consists of a space $\Gpdo$ (the unit space), a space $\cG$ of morphisms $s(g) \stackrel{g}{\To}r(g)$, continuous maps $s,r:\cG\To \Gpdo$ called \emph{source} and \emph{range} map, respectively, an inverse map $i:\cG\ni g\mto g^{-1}\in \cG$, and a partial product $\cG^{(2)}\ni (g,h)\mto gh\in \cG$, where $\cG^{(2)}$ is the set of composable pairs.

\begin{df}
A \emph{Real structure} on a groupoid $\grpd$ is a strict isomorphism $\tau: \cG \To \cG$ such that $\tau^2=\Id$. We say that $(\cG,\tau)$ is a Real groupoid. If $\cG$ is a Lie groupoid, $\tau$ is required to be a diffeomorphism.\\
We will often write $\bar{g}$ for $\tau(g)$, and we leave out $\tau$.  
Suppose $\cG$ and $\Ga$ are Real groupoids. A morphism $\phi:\Ga \To \cG$ is \emph{Real} if $\phi(\bar{\g})=\overline{\phi(\g)}, \forall \g\in \Ga$; in this case we say that $\phi$ is a morphism of Real groupoids. By $\cG_\RR$ we denote the subgroupoid of invariant elements of $\cG$.  
\end{df}

Observe that if the unit space is the point, then $\cG$ is a topological group $G$, and the unit element is a fixed point of the involution.

\begin{ex}
Any topological Real space $X$ in the sense of Atiyah can be thought of as a Real groupoid with unit space and space of morphisms identified with $X$; \emph{i.e.}, the operations are $s(x)=r(x)=x, \ x.x=x$, $x^{-1}=x$.
\end{ex}

\begin{ex}
Following the example above, the projective space $\PP\CC^n$ is a Real groupoid with respect to the coordinatewise complex conjugation; $(\PP\CC^n)_\RR$ is easily identified with the real projective space $\PP\RR^n$.	
\end{ex}

\begin{ex}
Let $X$ be a topological Real space. Let us consider the fundamental groupoid $\pi_1(X)$ over $X$ whose arrows from $x\in X$ to $y\in X$ are homotopy classes of paths (relative to endpoints) from $x$ to $y$ and the partial multiplication given by the concatenation of paths with fixed endpoints. The involution $\tau$ induces a Real structure on the groupoid as follows: if $[\g] \in \pi_1(X)$, we set $\tau([\g])$ the homotopy classes of the path $\tau(\g)$ defined by $\tau(\g)(t):=\tau(\g(t))$ for $t\in [0,1]$. 
\end{ex}
\medskip 

\begin{ex}[Orientifold groupoid]~\label{ex:orientifold}
Consider an orbifold $X\rtimes G$, where $G$ is a compact group with involution, acting equivariantly (\emph{i.e.} $\overline{x\cdot g}=\bar{x}\cdot \bar{g}$) and freely on the compact Real space $X$. Then $X\rtimes G$ is a Real groupoid with unit space $X$, source and range maps $s(g,x)=xg, r(x,g)=x$, partial product $(x,g)\cdot(xg,h)=(x,gh)$, and Real structure $(x,g)\mto (\bar{x},\bar{g})$.				
\end{ex}

Two Real structures $\tau$ and $\tau'$ on $\cG$ are said to be \emph{conjugate}, and we write $\tau\sim \tau'$, if there exists a groupoid automorphism $\phi: \cG \To \cG$ such that $\tau' = \phi \circ \tau \circ \phi^{-1}$.

\begin{ex}
Let $n\in \NN^\ast$. Suppose $\tau$ is a Real structure on the additive group $\RR^n$. Then, every $u\in \RR^n$ decomposes into a unique sum $v+w$ such that $\tau(v)=v$ and $\tau(w)=-w$. Indeed, $u=\frac{u+\tau(u)}{2}+\frac{u-\tau(u)}{2}$, so that $\RR^n=\ker (\frac{\id-\tau}{2})\oplus \operatorname{Im}(\frac{\id-\tau}{2})$, and with respect to this decomposition, $\tau$ is given by $\tau(v,w)=(v,-w)$. It then follows that there exists a unique decomposition $\RR^n=\RR^p\oplus \RR^q$ such that $\tau$ is determined by the formula \[\tau(x,y)=(\id_p\oplus(-\id_q))(x,y):=(x,-y),\]
for all $(x,y)=(x_1,\cdots, x_p,y_1,\cdots,y_q)\in \RR^p\oplus \RR^q$. \\
For each pair $(p,q)\in \NN$, we will write $\RR^{p,q}$ for the additive group $\RR^{p+q}$ equipped with the Real structure $(\id_p\oplus(-\id_q))$. 

Now, we define the Real space $\bfS^{p,q}$ as the stable subset (\emph{i.e.} invariant under the Real structure) of $\RR^{p,q}$ consisting of those $u\in \RR^{p+q}$ such that $\|u\|=1$. For $q=p$, $\bfS^{p,p}$ is clearly identified with the Real space $\bbS^p$ whose Real structure is given by the coordinatewise complex conjugation. Notice that $\bfS^{p,0}$ is the fixed point subset of $\bfS^{p,q}$.
\end{ex}

\begin{df}
Let $Z$ be a locally compact Hausdorff Real space, with Real structure $\tau: Z\ni z\mto\bar{z}:=\tau(z)\in Z$. A right Real action of $\cG$ on $Z$ consists of an action of $\cG$ on $Z$ which is compatible with the Real structures; \emph{i.e.} there is a continuous open map $\fs:Z\To \Gpdo$, and a continuous map $Z\ast\cG:=Z\times_{\fs,\Gpdo,r}\cG\To Z$, denoted by $(z,g)\mto z\cdot g$ (or just $zg$) such that
\begin{itemize}
\item[(a)] $\overline{z\cdot g}=\bar{z}\cdot\bar{g}$ for all $(z,g)\in Z \ast \cG$;
\item[(b)] $\overline{\fs(z)}=\fs(\bar{z}) $ for all $z\in Z$;
\item[(c)] $\fs(zg)=s(g)$;
\item[(d)] $z(gh)=(zg)h$ for $(z,g)\in Z \ast \cG$ and $(g,h)\in \cG^{(2)}$;
\item[(e)] $z\fs(z)=z$ for all $z\in Z$, where we identify $\fs(z)$ with its image in $\cG$ by the inclusion $\Gpdo \hookrightarrow \cG$. 	
\end{itemize}
We say that $Z$ is a (right) Real $\cG$-space. The Real action is \emph{free} (resp. \emph{proper}) if the Real map $Z\ast\cG \ni (z,g)\mto (z,zg)\in Z\times Z$ is one-to-one (resp. proper).
\end{df}

\begin{ex}
Let $(p,q)\in \NN$, and let us denote by $\ZZ^{p,q}$ the group $\ZZ^{p+q}=\ZZ^p\oplus \ZZ^q$ equipped with the Real structure inherited from that of $\RR^{p,q}$. Then, it is straightforward to see that the usual action of $\ZZ^{p+q}$ on $\RR^{p+q}$ by translation is compatible with the Real structures, so that $\RR^{p,q}$ is a Real $\ZZ^{p,q}$-space. 	
\end{ex}

\begin{df}
Let $\cG$ be a Real groupoid and $Y$ a Real space. A \emph{Real $\cG$-bundle} over $Y$ is a Real $\cG$-space with respect to a map $\fs:Z\To \Gpdo$, together with a Real open surjective map $\pi:Z\To Y$ such that $\pi(zg)=\pi(z)$ for all $(z,g)\in Z\ast \cG$. Such a $\cG$-bundle is \emph{principal} if 
\begin{itemize}
\item[(a)] $\pi$ is locally split (\emph{i.e.} it admits local sections), and 
\item[(b)] if the Real map $Z\ast\cG\ni (z,g)\mto (z,zg)\in Z\times_YZ$ is a homeomorphism.	
\end{itemize}
Finally, we say that $\pi:Z\To Y$ is \emph{locally trivial} if in addition
\begin{itemize}
 \item[(c)] for each $y\in Y$, there exist an open neighborhood $U$ of $y$ in $Y$, a continuous map $\vp_{U}:U\To \Gpdo$, and a homeomorphism $$\pi^{-1}(U)\To U\times_{\vp_U,\Gpdo,r}\cG,$$ such that  if we denote by $\tau_Z$ and $\tau_Y$ the Real structures of $Z$ and $Y$, respectively, then the following diagram commutes
 \begin{equation*}
 	\xymatrix{\pi^{-1}(U) \ar[d]^{\tau_Z} \ar[rr]^{\cong} && U\times_{\vp_U,\Gpdo,r}\cG \ar[d]^{\tau_Y\times \rho} \\
 	\pi^{-1}(\tau_Y(U)) \ar[rr]^{\cong} && \tau_Y(U)\times_{\vp_{_{\tau_Y(U)}},\Gpdo,r}\cG}	
 	\end{equation*}	
 \end{itemize} 	
\end{df}

\begin{ex}
The surjective Real map $\RR^{0,1}\To \bfS^{1,1}$, given by $(0,t)\mto (\cos 2\pi t,\sin 2\pi t)$, turns $\RR^{0,1}$ into a Real principal $\ZZ^{0,1}$-bundle over $\bfS^{1,1}$. Indeed, it is easy to check that the kernel of this map is $\ZZ^{0,1}$, so that $\bfS^{1,1}$ is isomorphiic (as a Real group) to the Real group $\TT^{0,1}:=\RR^{0,1}/\ZZ^{0,1}$. \\
\end{ex}

\begin{df}
A \emph{generalized Real morphism} $\xymatrix{\Ga \ar[r]^Z &\cG}$ from a Real groupoid $\Ga$ to a Real groupoid $\cG$ consists of a Real space $Z$, and two Real continuous open maps $$\xymatrix{ \Gamo & Z \ar[l]_\fr \ar[r]^\fs & \Gpdo},$$ giving to $Z$ the structures of a left Real $\Ga$-space (with respect to $\fr$) and a right Real $\cG$-space (with respect to $\fs$), such that 
\begin{itemize}
\item[(a)] the two actions commute: if $(z,g)\in Z \times_{\fs, \Gpdo,r}\cG$ and $(\g,z)\in \Ga\times_{s,\Gamo,\fr} Z$ we must have $\fs(\g z)= \fs(z)$, $\fr(zg)=\fr(z)$ so that $\g (zg)=(\g z)g$;
\item[(b)] $\fr: Z\To \Gamo$ is a locally trivial Real principal  $\cG$-bundle.		
	\end{itemize}
		
Such a $Z$ is said to be a \emph{Morita equivalence} if in addition 
\begin{itemize}
	\item[(c)] $\fs:Z\To \Gpdo$ is a locally trivial Real principal $\Ga$-bundle. In that case, we say that the Real groupoids $\Ga$ and $\cG$ are \emph{Morita equivalent}, and we write $\Ga\sim_Z \cG$.	
	\end{itemize}	
\end{df}

\begin{ex}
Let $f: \Ga \To \cG$ be a strict Real morphism. Let us consider the fibred product $Z_f:= \Gamo \times_{f,\Gpdo, r} \cG$ and the maps $\fr: Z_f \To \Gamo, \ (y,g)\mto y$ and $\fs: Z_f \To \Gpdo, \ (y,g)\mto s(g)$. For $(\g, (y,g))\in \Ga \times_{s,\Gamo, \fr} Z_f)$, we set $\g\cdot(y,g):=(r(\g),f(\g)g)$ and for $((y,g),g')\in Z_f\times_{\fs,\Gpdo,r} \cG$ we set $(y,g).g':= (y,gg')$. Using the definition of a strict morphism, it is easy to check that these maps are well defined and make $Z_f$ into a generalized morphism from $\Ga$ to $\cG$. Furthermore, these maps are compatible with the Real structure $\tau$ on $Z_f$ defined by $\tau(y,g):=(\bar{y},\bar{g})$; so that $Z_f$ is generalized Real morphism.
\end{ex}

Given a Morita equivalence $\Ga\sim_Z \cG$ of Real groupoids, its inverse, denoted by $(Z^{-1},\tau)$, is $(Z,\tau)$ as Real space, and if $\flat: Z\To Z^{-1}$ is the identity map, the left Real $\cG$-action on $(Z^{-1},\tau)$ is given by $g\cdot\flat(z):=\flat(z\cdot g^{-1})$, and the right Real $\Ga$-action is given by $\flat(z)\cdot\g:=\flat(\g^{-1}\cdot z)$; $Z^{-1}$ is then a generalized Real morphism from $\cG$ to $\Ga$. 	

\section{Real \v{C}ech cohomology}

In this section we recall the definition of \emph{Real \v{C}ech cohomology} defined in the author's thesis~\cite{Mou}. First recall from~\cite{Tu1} that the cohomology groups $\check{H}^\ast (\cG_\bullet,\cF^\bullet)$ are defined as the \v{C}ech cohomology groups of the simplicial spaces $\cG_\bullet$ with $n$-simplex $$\cG_n:=\cG^{(n)}:=\{\overrightarrow{g}=(g_1,\cdots, g_n)\in \cG^n\mid s(g_1)=r(g_2), \cdots, s(g_{n-1})= r(g_n)\},$$ and where $\cF^\bullet$ is an (pre-)simplicial (pre-)sheaf over $\cG_\bullet$ (cf. ~\cite{Tu1}). The $n$-cochains on $\cG$ with respect to a pre-simplicial open cover  $\cU_\bullet=(\{U^n_{\lambda}\}_{\lambda\in \Lam_n})_{n\in \NN}$ of $\cG_\bullet$ are the elements of the group 
\[C^n(\cU_\bullet,\cF^\bullet):=\prod_{\lambda \in \Lam_n}\cF^n(\cU_\lambda^n).\]
Moreover, the differential $d^n:C^n(\cU_\bullet,\cF^\bullet)\To C^{n+1}(\cU_\bullet,\cF^\bullet)$ is defined by the formula $$(d^nc)_{\lambda\in \Lam_{n+1}}:=\sum_{k=0}^{n+1}(-1)^k \tilde{\ve}_k^\ast(c_{\tilde{\ve}_k(\lambda)}),$$ where $\ve_k:[n]\To [n+1]$ is the unique increasing injective map that avoids $k$.

Now, let $\cG$ be a Real groupoid. Then $\cG_n$ is a Real space with respect to the involution $\tau_n(\overrightarrow{g}):=(\bar{g}_1,\cdots,\bar{g}_n)$. Suppose $\cF^\bullet$ is the abelian sheaf such that for all $\lambda\in \Lam_n$, $\cF^n(U_\lambda^n)=\cC(U^n_\lambda,S)$, where $S$ is an abelian group equipped with a Real structure $s\mto \bar{s}$. Then, $\cF^\bullet$ is a Real sheaf over $\cG_\bullet$ in the sense that for all $n\in \NN$, and for all pre-simplicial Real open cover $\cU_\bullet$ of $\cG_\bullet$, we have a commutative diagram 
\[\xymatrix{\cF^n(U^n_\lambda) \ar[d]^{\sigma_{U^n_\lambda}} \ar[r] & \cF^n(V^n_\lambda) \ar[d]^{\sigma_{V^n_\lambda}} \\
\cF^n(U^n_{\bar{\lambda}}) \ar[r] & \cF^n(V^n_{\bar{\lambda}})
} \]
whenever $V^n_\lambda\subset U^n_\lambda$, where the horizontal arrows are the restriction maps, and the vertical ones are defined by $\cC(U^n_\lambda,S)\ni f \mto \bar{f}\in \cC(U^n_{\bar{\lambda}},M)$, with $\bar{f}(\tau_n(\overrightarrow{g})):=\overline{f(\overrightarrow{g})}$ for $\overrightarrow{g}\in U^n_\lambda$.

\begin{df}
Let $\cG$ and $\cF$ be as above. Given a pre-simplicial Real open cover $\cU_\bullet$ of $\cG_\bullet$, we will write for simplicity $C^n(\cU_\bullet,S):=C^n(\cU_\bullet,\cF^\bullet)$. Set
\[CR^n(\cU_\bullet,S):=\{(c_\lambda)_{\lambda\in \Lam_n}\in C(\cU_\bullet,S) \mid c_{\bar{\lambda}}=\overline{c_\lambda}, \forall \lambda\in \Lam_n\}.\]	
Then, an element in $CR^n(\cU_\bullet,S)$ is called a \emph{Real $n$-cochain with coefficient in $S$} with respect to $\cU_\bullet$. Then the complex of abelian groups $CR^\ast(\cU_\bullet,S)$ is closed under the differential $d: C^\ast(\cU_\bullet,S)\To C^{\ast+1}(\cU_\bullet,S)$. We define the $n^{\text{th}}$ \emph{Real cohomology group} of $\cG$ with coefficients in $S$ relative to $\cU_\bullet$ as 
\[HR^n(\cU_\bullet,S):= \frac{ZR^n(\cU_\bullet,S)}{BR^n(\cU_\bullet,S)}:=\frac{\ker d^n}{\operatorname{Im} d^{n-1}}.\] 
Finally, we define the \emph{Real \v{C}ech cohomology groups} of $\cG$ with coefficients in the Real group $S$ by
\[\check{H}R^\ast(\cG_\bullet,S):=\varinjlim{HR^\ast(\cU_\bullet,S)},\] where the injective limit is taken over the collection of all pre-simplicial Real open covers of $\cG_\bullet$ (see ~\cite{Tu1}).
\end{df}

\begin{exs}~\label{exs:HR0-HR1}
1) \emph{(The group $\check{H}R^0$).} From the construction, $\check{H}R^0(\cG_\bullet,S)\cong\Ga_{\text{inv}}(\Gpdo,S)_\RR$, where the latter is the group consisting of continuous functions $f:\Gpdo\To S$ such that $f(\bar{x})=\overline{f(x)}, \forall x\in \Gpdo$, and which are $\cG$-invariant; i.e. $f(s(g))=f(r(g)), \forall g\in \cG$. For instance, if $\cG_\RR=\cG$, then $\check{H}R^0(\cG_\bullet,S)\cong \Ga_{\text{inv}}(\Gpdo,S_\RR)$.\\

2) \emph{(The group $\check{H}R^1(\cG_\bullet,S)$).} It is easily checked that $\check{H}R^1(\cG_\bullet,S)$ is isomorphic to the abelian group $\Hom_{\RG}(G,S)$ of isomorphism classes of generalized Real morphisms from $\cG$ to $S$. In other words, $\check{H}R^1(\cG_\bullet,S)$ classifies the isomorphism classes of Real $\cG$-equivariant principal $S$-bundles over the Real space $\Gpdo$. For instance, with respect to the trivial Real structure on $\cG$, we have $\check{H}R^1(\cG_\bullet,S)\cong \check{H}^1(\cG_\bullet,S_\RR)$.\\

3) \emph{(The Real Picard group).} A particular case of the above is when $S=\uc$ equipped with the complex conjugation. In this case, $\check{H}R^1(\cG_\bullet,\uc)$ is isomorphic to the \emph{Real Picard group} $\PicR(\cG)$ of $\cG$ defined as the set of isomorphism classes of $\cG$-equivariant Real line bundles over $\Gpdo$ equipped with $\cG$-invariant hermitian metrics compatible with the Real structures. In particular, if $\tau$ is trivial, then $\PicR(\cG)\cong \check{H}^1(\cG_\bullet,\ZZ_2)$ (since $\uc_\RR\cong \ZZ_2$). For example, for $\cG=\xymatrix{G\dar[r] & \cdot }$ with the trivial Real structure, we have $\PicR(\cG)\cong \Hom(G,\ZZ_2)$.\\

4) \emph{(Field strenght of Neveu-Schwarz $B$-fields).} Let $M$ be a smooth compact $10$-manifold with involution $\tau:M\To M$. Then $\check{H}R^n(M_\bullet,\ZZ^{0,1})$ consists of all $n$-forms $\omega \in \Omega^n(M)$ such that $\tau^\ast \omega=-\omega$. for instance, in a superstring theory background, an element in $\check{H}R^3(M_\bullet,\ZZ^{0,1})$ may be thought of as the field strenght of a NS $B$-field (cf.~\cite{DFM}); this justifies the terminology used in Section~\ref{B-fields}. 
\end{exs}

\section{Real graded central extensions}

\begin{df}
A Real graded $\uc$-groupoid over $\cG$ consists of a Real groupoid $\tilde{\cG}$ with unit space $\Gpdo$, a Real $\uc$-principal bundle $\pi:\tilde{\cG}\To\cG$, and a Real morphism $\del:\cG\To \ZZ_2$, where $\ZZ_2$ is given the trivial involution.

Morita equivalence and tensor product of such are defined as in the usual case (cf. Freed-Hopkins-Teleman~\cite{FHT}, and Tu~\cite{Tu}), plus the compatibility with the involutions involved. The abelian group thus obtained is denoted by $\Rext(\cG,\uc)$.
A Real graded $\uc$-central extension of $\cG$ is a Real graded $\uc$-groupoid $(\wGa,\del)$ over a Real groupoid $\Ga$ Morita equivalent to $\cG$. Such an extension will be written as a triple $(\wGa,\Ga,\del)$. The set of Morita equivalent classes of Real graded $\uc$-central extension of $\cG$ is an abelian group $\wRExt(\cG,\uc)$ whose zero element is the class of the trivial extension $\uc\To \cG\times\uc \To \cG$ with trivial grading.  	
\end{df}

\begin{ex}~\label{ex:ext-PU}
Let $\hat{\cH}:=l^2(\NN)\oplus l^2(\NN)$ be the graded infinite-dimensional Hilbert space whose grading is given by the unitary operator $u_0:=\begin{pmatrix}0&\id\\ \id&0\end{pmatrix}$. Put $J_\RR=\begin{pmatrix}0&bar \\ bar &0\end{pmatrix}$, where $bar:l^2(\NN)\To l^2(\NN)$ is the complex conjugation with respect to the canonical basis of $l^2(\NN)$. Then $J_\RR$ is a Real structure commuting with the grading. For $i =0,1 \mod 2$, let $\U^i(\hat{\cH})$ denote the set of unitaries of degree $i$. Next, let the group $\wU(\hat{\cH}):=\U^0(\hat{\cH})\sqcup\U^1(\hat{\cH})$ be equipped with the Real structure given by $Ad_{J_\RR}$. Notice that scalar multiplication by elements of $\uc$ defines a Real $\uc$-action on $\wU(\hat{\cH})$ and that $Ad_{J_\RR}$ passes through the quotient $\wPU(\hat{\cH}):=\wU(\hat{\cH})/\uc$. Therefore, we have the Real graded $\uc$-central extension $(\wU(\hat{\cH}),\wPU(\hat{\cH}),\partial)$ of the Real groupoid $\xymatrix{\wPU(\hat{\cH})\dar[r]& \cdot}$, where $\partial([u])\in\{0,1\}$ is the degree of $u$.	
\end{ex}

\medskip 

\begin{thm}~\label{thm:extensions}
$\wRExt(\cG,\uc)\cong \check{H}R^1(\cG_\bullet,\ZZ_2)\ltimes \check{H}R^2(\cG_\bullet,\uc)$, where the action is induced from the Real inclusion $\ZZ_2=\{0,1\}\hookrightarrow \uc$ obtained by identifying $1$ with $1+i0\in \uc$. Furthermore, if $\cG$ is proper, then $$\wRExt(\cG,\uc)\cong \check{H}R^1(\cG_\bullet,\ZZ_2)\ltimes \check{H}R^3(\cG_\bullet,\ZZ^{0,1}).$$	
\end{thm}

The proof is almost the same as that of~\cite[Corollary 2.25]{FHT}.

\begin{rem}
Note that $\cG$ is a space with the trivial involution, then $\wRExt(X,\uc)$ is actually the graded analog of the group of stable isomorphism classes of \emph{real bundle gerbes} defined by Mathai-Murray-Stevenson in~\cite{MMS}.	
\end{rem}


\section{$B$-fields over Real groupoids}~\label{B-fields}

Let $\wK_{ev}=\cK(\hat{\cH})$ be the $\ZZ_2$-graded $\cstar$-algebra of compact operators, with grading given by $Ad_{u_0}$ (cf. Example~\ref{ex:ext-PU}). Let $\wK_{od}:=\cK(l^2(\NN))\oplus \cK(l^2(\NN))$ with the odd grading $(T_1,T_2)\mto (T_2,T_1)$. A $\ZZ_2$-graded $\cstar$-algebra $A$ is said elementary of parity $0$ (resp. $1$) if it is isomorphic to $\wK_{ev}$ (resp. $\wK_{od}$).

Recall that a \emph{Real $\cstar$-algebra} is a $\cstar$-algebra $A$ endowed with a \emph{Real structure}; \emph{i.e.}, an involutory conjugate linear automorphism $\sigma:A\To A$. In this section we are concerned, among other things, with Real structures on $\wK_{ev}$ and $\wK_{od}$. We classify these as follows. First observe that a Real structure on $\wK_{ev}$ is neccessarily in the form $Ad_J$, for a conjugate linear homogeneous unitary $J:\hat{\cH}\To \hat{\cH}$ such that $J^2=\pm 1$. Say a Real graded $\cstar$-algebra $A$ is of type $[0;\ve,\eta]$, if it is isomorphic (in an equivariant way) with $\wK_{ev}$ equipped with a Real structure $Ad_J$ such that $\partial(J)=\ve\in \{0,1\}$ and $J^2=\eta1, \eta\in \{\pm\}$. Similarly, a Real structure on $\wK_{od}$ is either in the form $\sigma^0=\begin{pmatrix}Ad_J & 0\\ 0& Ad_J\end{pmatrix}$ or in the form $\sigma^1=\begin{pmatrix}0&Ad_J\\ Ad_J & 0\end{pmatrix}$, where $J:l^2(\NN)\To l^2(\NN)$ is a conjugate linear unitary such that $J^2=\pm1$. We say that $A$ is of type $[1;i,\ve]$ if it is isomorphic to $\wK_{od}$ equipped with a Real structure of the form $\sigma^i, i=0,1$ with $J^2=\ve1, \ve\in \{\pm\}$. We summarize these by a table:

\begin{table}[!h]
\centering

\begin{tabular}{|cc|} 
 \hline Parity $0$  & Parity $1$  \\ \hline
  $\what{\cK}_0:=[0;0,+]$    & $\what{\cK}_1:=[1;0,+]$  \\ 
  $\what{\cK}_2:=[0;1,+]$    & $\what{\cK}_3:=[1;1,-]$  \\ 
  $\what{\cK}_4:=[0;0,-]$    & $\what{\cK}_5:=[1;0,-]$   \\ 
  $\what{\cK}_6:=[0;1,-]$    & $\what{\cK}_7:=[1;1,+]$   \\ \hline

\end{tabular}

\end{table} 	

Note that types do add up under graded tensor product: $\wK_p\hat{\otimes}\wK_q=\wK_{p+q}$.

\begin{df}
Let $\grpd$ be a Real groupoid. A \emph{Real graded Dixmier-Douady bundle} over $\cG$ consists of a Real space $\cA$, an equivariant continuous surjective map $p:\cA\To \Gpdo$ defining a continuous bundle of graded $\cstar$-algebras such that the induced map $\tau_x:\cA_x\To \cA_{\bar{x}}$ is a conjugate linear graded isomorphism, a family $\al$ of isomorphisms of graded $\cstar$-algebras $\al_g:\cA_{s(g)}\To \cA_{r(g)}$ such that $\al_{gh}=\al_g\circ \al_h$ whenever the product makes sense and  such that $\al_{\bar{g}}(\bar{a})=\overline{\al_g(a)} \forall a\in \cA_{s(g)}$, a family of $\ZZ_2$-graded isomorphisms $h_x: \cA_x\To \wKK$ (where $\wKK$ is either $\wK_{ev}$ or $\wK_{od}$ and is the same for every $x$), and a Real structure $\sigma$ on $\wKK$ such that the following commutes 
\begin{equation}
	\xymatrix{\cA_x\ar[r]^{h_x} \ar[d]^{\tau_x} & \wKK \ar[d]^\sigma \\ \cA_{\bar{x}} \ar[r]^{h_{\bar{x}}} & \wKK}	
	\end{equation}
over all $x\in \Gpdo$.
The family $\al$ is called the Real $\cG$-action. Such $\cA$ is said of \emph{type $p \mod 8$} if $\wKK$ is a Real graded elementary $\cstar$-algebra of type $\wK_p$ with respect to $\sigma$. The triple $(\cA,\al,p)$ will be called a \emph{$B$-field} of type $p$ over the Real groupoid $\cG$; for the sake of simplicity,  the $B$-field will be represented by the bundle $\cA$.		
\end{df}

There is a canonical $B$-field $\wK_\cG$ of type $0$ constructed as follows. Let $\mu=\{\mu^x\}_{x\in \Gpdo}$ be a Real Haar system (\emph{i.e.} $\tau^\ast\mu =\mu$); this is easily seen to always exist. For each $x\in \Gpdo$, consider the graded Hilbert space $\hat{\cH}_{x,\cG}:=L^2(\cG^x;\hat{\cH})$ with scalar product
\begin{equation*}
	\tlip<\xi,\eta>(x):= \int_{\cG^x}\<\xi(g),\eta(g)\>_\CC d\mu_\cG^x(g), \ {\rm for\ } \xi,\eta\in L^2(\cG^x;\hat{\cH})\cong L^2(\cG^x)\hat{\otimes}\hat{\cH}.
\end{equation*} 
Let $\hat{\cH}_{\cG}:=\coprod_{x\in X}\hat{\cH}_x$ be equipped with the action $$g\cdot(s(g),\vp\hat{\otimes}\xi):=(r(g),(\vp\circ g^{-1})\hat{\otimes}\xi)\in \what{\cH}_{r(g)}.$$ Define the Real structure on $\what{\cH}_{\cG}$ by $$(x,\vp \hat{\otimes}\xi)\mto (\bar{x},\tau(\vp)\hat{\otimes}J_{\RR,0}(\xi)),$$ where for $\vp\in L^2(\cG^x)$ and $g\in \cG^x$, $\tau(\vp)(g):=\overline{\vp(\bar{g})}$. Then one shows that there exists a unique topology on $\hat{\cH}_\cG$ such that the canonical projection $\what{\cH}_{\cG}\To \Gpdo$ defines a locally trivial Real graded Hilbert $\cG$-bundle.

Now, let $\what{\cK}_{x}:=\cK(\what{\cH}_{x})$ be equipped with the operator norm topology, and put $$\what{\cK}_{\cG}:=\coprod_{x\in X}\what{\cK}_{x}$$ together with the Real structure given by $\overline{(x,T)}:=(\bar{x},\bar{T})$, where $\bar{T}\in \what{\cK}_{\bar{x}}$ is defined by $\bar{T}(\vp\hat{\otimes}\xi):=T(\tau(\vp)\hat{\otimes}Ad_{J_\RR}(\xi))$ for any $\vp\hat{\otimes}\xi \in \what{\cH}_{\bar{x}}$. Next, define the Real $\cG$-action $\theta$ on $\what{\cK}_{\cG}$ by 
\[ \theta_g(s(g),T):=(r(g),gTg^{-1}). 
\]
Then the canonical projection $\what{\cK}_{\cG}  \To X, \ (x,T)\mto x$ makes $\wK_\cG$ into a $B$-field of type $0$ over $\cG$.

Two $B$-fields $\cA$ and $\cB$ are said \emph{Morita equivalent} if they are of the same type and if $\cA\hat{\otimes}\wK_\cG\cong \cB\hat{\otimes}\wK_\cG$. Denote by $\wRBr_p(\cG)$ the set of Morita equivalence classes of $B$-fields of type $p \mod 8$, and define the \emph{Real graded Brauer group} as $\wRBr_\ast(\cG):=\bigoplus_{p=0}^7 \wRBr_p(\cG)$.

\begin{pro}
$\wRBr_\ast(\cG)$ is an abelian group under the operations of graded tensor product, and inversion given by considering conjugate bundles. The identity element of $\wRBr_\ast(\cG)$ is given by the class of the trivial bundle $\Gpdo\times \CC\To \Gpdo$, with Real $\cG$-action  $\al_g (s(g),\lambda)):= (r(g),\lambda)$ and Real structure $(x,\lambda)\mto (\bar{x},\bar{\lambda})$.	
\end{pro}

 Note that $\wRBr_0(\cG)$ is a subgroup of $\wRBr_\ast(\cG)$. Furthermore, the map sending a $B$-field $\cA$ into the pair consisting of its type $p$ and the $B$-field $\cA\hat{\otimes}\wK_{8-p}$ of type $0$ induces an isomorphism of abelian groups 
\begin{equation}~\label{eq:isom-BrR_vs_BrR_0}
\wRBr_\ast(\cG)\cong \check{H}R^0(\cG_\bullet,\ZZ_8)\oplus \wRBr_0(\cG).	
\end{equation}

\section{Cohomological interpretation of $\wRBr_\ast(\cG)$}

The purpose of this section is to show the following.

\medskip 

\begin{thm}~\label{thm:main}
Let $\grpd$ be a Real groupoid. Then 
\[\wRBr_\ast(\cG)\cong \check{H}R^0(\cG_\bullet,\ZZ_8)\oplus (\check{H}R^1(\cG_\bullet,\ZZ_2)\ltimes \check{H}R^2(\cG_\bullet,\uc)).\]	
\end{thm}

To do so, we just need to show $\wRBr_0(\cG)\cong \wRExt(\cG,\uc)$, thanks to Theorem~\eqref{thm:extensions} and the isomorphism~\ref{eq:isom-BrR_vs_BrR_0}. 

Note that the isomorphism mentioned in Example~\ref{exs:HR0-HR1} 2) remains true even for a non-abelian Real group $S$; in that case the set $\check{H}R^1(\cG_\bullet,S)\cong \Hom_{\RG}(\cG,S)$ is however not a group since the sum of two cocycles is not necessarily a cocycle. Nevertherless, the sets $\check{H}R^1(\cG_\bullet,\wPU(\hat{\cH}))$ and $\check{H}R^1(\cG_\bullet,\wU^0(\hat{\cH}))$ are abelian monoids under the operations of taking the graded tensor product of cocycles (and considering a fixed isomorphism of Real graded Hilbert spaces $\hat{\cH}\hat{\otimes}\hat{\cH}\cong \hat{\cH}$). The corresponding operation in $\Hom_{\RG}(\cG,\wPU(\hat{\cH}))$ and $\Hom_{\RG}(\cG,\wU^0(\hat{\cH}))$ is written additively.

\begin{lem}(compare~\cite{Tu}).
Let $\PP_0:=\coprod_{x\in \Gpdo}\wPU(\hat{\cH},\hat{\cH}_x)\To \Gpdo$ be the Real $\wPU(\hat{\cH})$-principal bundle	associated to $\wK_\cG$. Then $$\wRBr_0(\cG)\cong \Hom_{\RG}(\cG,\wPU(\hat{\cH}))_{stable},$$
where $\Hom_{\RG}(\cG,\wPU(\hat{\cH}))_{stable}:=\{P\in \Hom_{\RG}(\cG,\wPU(\hat{\cH})) \mid P+\PP_0=P\}$.
\end{lem}

\begin{proof}
If $P:\cG\To \wPU(\hat{\cH})$ is a generalized Real morphism, then $\cA:=P\times_{\wPU(\hat{\cH})}\wK_0\To \Gpdo$ is a $B$-field of type $0$. Conversely, if $\cA$ is a $B$-field of type $0$, then we obtain $P$ by setting $P_x:=\Isom^{(0)}(\wK_0,\cA_x)$, the latter being the set of graded ${}^\ast$-isomorphisms. We then have a split-exact sequence of abelian monoids 
\[0\To \Hom_{\RG}(\cG,\wU^0(\hat{\cH})) \stackrel{pr}{\To} \Hom_{\RG}(\cG,\wPU(\hat{\cH})) \To \wRBr_0(\cG) \To 0\]
where the second arrow is induced by the canonical projection.	
\end{proof}

\begin{lem}
$\Hom_{\RG}(\cG,\wPU(\hat{\cH}))_{stable}\cong \wRExt(\cG,\uc)$.	
\end{lem}

We only explain the map $\Hom_{\RG}(\cG,\wPU(\hat{\cH}))\To \wRExt(\cG,\uc)$. Given $P:\cG\To \wPU(\hat{\cH})$, we may view it as a morphism of Real groupoids $p:\Ga\To \wPU(\hat{\cH})$ where $\Ga$ is a Real groupoid Morita equivalent to $\cG$. Hence we obtain a Real graded $\uc$-central extension $(\wGa,\Ga,\del)$ by pulling back the extension $(\wU(\hat{\cH}),\wPU(\hat{\cH}),\partial)$ (cf. Example~\ref{ex:ext-PU}) along $p$. 

Now the proof of Theorem~\ref{thm:main} is straightforward.
\begin{rem}
It is easy to verify that if $\cG$ is a CW-complex $X$ with the trivial involution, then we recover Donovan-Karoubi's \emph{graded orthogonal Brauer group} $\operatorname{GBrO}(X)\cong H^0(X,\ZZ_8)\oplus (H^1(X,\ZZ_2)\ltimes H^2(X,\ZZ_2))$.	
\end{rem}

\begin{ex}
$\wRBr_\ast(pt)=\ZZ_8$.	
\end{ex}


\section{Twisted $KR$-theory}

For a $B$-field $\cA$ over the Real groupoid $\cG$, the convolution ${}^\ast$-algebra $\cC_c(\cG;s^\ast\cA)$ is given the fibrewise grading of $\cA$, and comes equipped with the Real structure $\xi\mto \bar{\xi}$, where for $\xi\in \cC_c(\cG;s^\ast\cA)$, $\bar{\xi}(g):=\overline{\xi(\bar{g})}, \forall g\in \cG$. This induces a Real structure on the graded reduced $\cstar$-algebra $\cA\rtimes_r\cG:=\cstar_r(\cG;s^\ast\cA)$. 

\begin{df}
Twisted $KR$-theory is defined in terms of Kasparov's bivariant Real $K$-theory~\cite{Kas81} by setting
\[KR^{q-p}_\cA(\cG):= \left\{\begin{array}{ll}KKR(\CC l_{p-q,0},\cA\rtimes_r\cG), & {\rm if\ } p\ge q;\\ KKR(\CC l_{q-p,0},\cA\rtimes_r\cG), & {\rm if\ } p\le q,\end{array}\right.\]
where the complex Clifford algebras $\CC l_{k,l}$ carry out the Real structure induced from the involution $$\CC^{k+l}\ni (x_1,...,x_k,y_1,...,y_l)\mto (\bar{x}_1,...,\bar{x}_k,-\bar{y}_1,...,-\bar{y}_l)\in \CC^{k+l}.$$ 		
	\end{df}
	
We deduce immediately the following

\begin{pro}[Bott periodicity]
	$KR^{-8-q-p}_\cA(\cG)\cong KR^{q-p}_\cA(\cG)$.		
		\end{pro}

A particular case is when $\cG$ consists of two pieces $\cG_1$ and $\cG_2$ with involution going from $\cG_1$ to $\cG_2$ and \emph{vice versa}. In that case, we recover ordinary twisted complex $K$-theory: giving $\cA$ is equivalent to giving a $\cG_1$-equivariant bundle of complex elementary $\cstar$-algebras $\cA_1$ on the unit space $X_1$ of $\cG_1$, together with its conjugate bundle over the other piece. Therefore $\wRBr_\ast(\cG)\cong \wBr_\ast(\cG_1)$, where the latter is the graded complex Brauer group~\cite{DK,FHT,Tu}. What is more, we have

\begin{lem}
In the situation above, $KR^\ast_\cA(\cG)\otimes \ZZ[{1\over 2}]\cong K^\ast_{\cA_1}(\cG_1)\otimes\ZZ[{1\over 2}]$.			
		\end{lem}
\begin{ex}
 The Real space $\bfS^{0,1}$ is the set $\{-1,+1\}$ together with the involution $-1\mto +1$. Thus, $\wRBr(\bfS^{0,1})\cong \wBr(pt)\cong \ZZ_2$ since bundles of graded elementary complex $\cstar$-algebras on the point are trivial, and are therefore characterized by their parity. Moreover, if $\cA\in \wRBr(\bfS^{0,1})$ is realized by $i\in \ZZ_2$, we get (after tensoring with $\ZZ[{1\over 2}]$):
\[KR_\cA^{q-p}(\bfS^{0,1})=\left\{ \begin{array}{ll}\ZZ, & {\rm if\ } q-p-i = 0 \mod 2;\\ 0, & {\rm if\ } q-p-i =1\mod 2. \end{array}\right.\] 
\end{ex}

More generally, an element of $\wRBr_\ast(\cG)$ being obviously in the graded complex Brauer group $\wBr_\ast(\cG)$, the twisted complex $K$-theory groups $K^\ast_\cA(\cG)$ are well defined and are $2$-periodic. Moreover, from a slightly tricky construction of an involution on the Kasparov KK-group $KK(\CC l_{k,l},\cA\rtimes_r\cG)$, we obtain the following decomposition of twisted complex $K$-groups of $\cG$.  

\begin{thm}
Suppose $\cG$ is equipped with a non-trivial Real structure. Then for $\cA\in \wRBr_\ast(\cG)$ and $j\in \ZZ$, we have $$K^{-j}_\cA(\cG)\otimes\ZZ[{1\over 2}] \cong \left(KR^{-j}_\cA(\cG)\oplus KR^{-j-2}_\cA(\cG)\right)\otimes\ZZ[{1\over 2}].$$	
\end{thm} 

We shall point out that twisted $KR$-theory of \emph{orientifold groupoids} (Example~\ref{ex:orientifold}) admits a more concrete picture. Roughly speaking, it can be expressed in terms of equivariant Fredholm operators. First we fix for every difference $-n=q-p\in \ZZ$ degree $1$ operators $\ve_1,...,\ve_q,e_1,...,e_p\in \cL(\hat{\cH})$ such that $\ve_i^2=1,\ve^\ast=\ve_i,i=1,..,q,e_j^2=-1,e_j^\ast=-e_j,j=1,...,p$ and $ee'=-e'e, \forall e \neq e' \in \{\ve_i,e_j\}$. Next define $\hat{\mathscr{F}}^{n}$ as the subset of $\cL(\hat{\cH})$ of degree $1$ operators $F$ such that $F^2-\id, F^\ast-F\in \wK_0$, and $Fe=-eF, \forall e\in \{\ve_i,e_j; i=1,...,q,j=1,...,p\}$. Similarly, form $\wU(\hat{\cH})_n$ the Real subgroup of $\wU(\hat{\cH})$ of those $u$ that commute with the $\ve_i$ and $e_j$. Then let $\wPU(\hat{\cH})_n$ be the image of $\wU(\hat{\cH})_n$ in the quotient $\wPU(\hat{\cH})$. Then 

\begin{thm}
Suppose a $B$-field $\cA$ over the orientifold groupoid $X\rtimes G$ is represented by a Real graded Dixmier-Douady bundle of type $0$. Let $P\in \Hom_{\RG}(\cG,\wPU(\hat{\cH}))_{stable}$ be its corresponding generalized morphism. Then for $n=p-q\in \ZZ$, $$KR^{-n}_\cA(X\rtimes G)\cong \left[ P^{(n)}/G,\hat{\mathscr{F}}^n\right]^{\wPU(\hat{\cH})_n}_R,$$ where $P^{(n)}\To X$ is the Real $G$-equivariant subbundle of $P$ whose fibre at $x\in X$ consists of all elements $\vp \in \Isom^{(0)}(\wK_0,\cA_x)$ such that for all $a\in \cA_x$, the operator $\vp^{-1}(a)$ commutes with the $\ve_i$ and $e_j$. Here $[\cdot,\cdot]^{\wPU(\hat{\cH})_n}_R$ means $\wPU(\hat{\cH})_n$-equivariant functions that are compatible with the Real structures.
\end{thm}



\end{document}